\documentclass[11pt]{amsart}

\usepackage[margin=1in]{geometry}
\usepackage{amsmath, amssymb, amsthm}
\usepackage{mathrsfs}
\usepackage{array}
\usepackage{enumitem}
\usepackage[dvipsnames]{xcolor}
\usepackage{graphicx}
\usepackage{fancyhdr}
\usepackage{tikz}
\usepackage{caption}
\usepackage[colorlinks=true, urlcolor=NavyBlue, linkcolor=NavyBlue, citecolor=NavyBlue]{hyperref}
\usepackage[framemethod=default]{mdframed}
\usepackage{framed}
\usepackage{pict2e}
\usepackage{wrapfig}

\usepackage{wrapfig}
\usepackage{mdframed}
\usepackage{float}

\usepackage{tikz,lipsum,lmodern}
\usepackage[most]{tcolorbox}

\newtheorem{theorem}{Theorem}
\newtheorem*{theorem*}{Theorem}

\newtheorem{proposition}[theorem]{Proposition}

\newtheorem{remark}[theorem]{Remark}

\theoremstyle{definition}

\newtheorem*{definition*}{Definition}

\newcommand{\Q}{\mathbb Q}

\newcommand{\Z}{\mathbb Z}

\newcommand{\bmat}[1]{\begin{bmatrix} #1 \end{bmatrix}}

\newcommand{\red}[1]{\textcolor{OrangeRed}{#1}}

\makeatletter
\renewcommand{\@seccntformat}[1]{}
\makeatother

\begin{document}
\title{Obstructing unknotting number two for some alternating knots}

\author[]{Justin Gebel}
\address{Department of Mathematics and Statistics 
    McMaster University, Hamilton, ON, L8S 4L8}
\email{gebelj@mcmaster.ca}

\author[]{William Prangley}
\address{Department of Mathematics and Statistics 
    McMaster University, Hamilton, ON, L8S 4L8}
\email{pranglew@mcmaster.ca}

\thanks{Justin Gebel was partially supported by an NSERC Canada Graduate Research Scholarship - Master's grant. William Prangley was partially supported by an Ontario Graduate Scholarship grant. This work was made possible by the facilities of the Shared Hierarchical 
Academic Research Computing Network (SHARCNET:www.sharcnet.ca) and Digital Research Alliance of Canada (https://alliancecan.ca/en).}

\begin{abstract}
    We apply the $d$-invariants of Ozsv\'ath-Szab\'o and a theorem due to Brendan Owens to determine the unknotting numbers of 194 alternating knots with 13 or fewer crossings.
\end{abstract}

\maketitle

\noindent\textbf{Introduction.} The unknotting number $u(K)$ of a knot $K$ is the minimal number of crossing changes needed to turn $K$ into the unknot, taken across all diagrams of $K$. The interest in computing unknotting numbers has recently been rekindled by the landmark papers \cite{BHog} and \cite{BH} from Brittenham and Hermiller, which exhibit and explore the non-additivity of $u(K)$ under connected sum. Following their terminology, if two knots $K$ and $K'$ are such that $u(K \# K')<u(K)+u(K')$, then the knots $K$ and $K'$ are said to be \textbf{symbionts} for each other. Whether there exists a knot which does not possess a symbiont is unknown, and provides a deeply interesting open problem. This project grew out of a failed attempt to show that $u(10_6)=3$, which would give the trefoil knot a symbiont and demonstrate this property for all torus knots (see \cite[Theorem 1.3, Corollary 4.2]{BH}). Although $u(10_6)$ is still unknown, we are, however, pleased to present 194 new unknotting numbers, which were computed using techniques of Owens and Ozsv\'ath-Szab\'o.

\begin{theorem}
    The following 11 and 12 crossing knots have $u(K)=3$: 11a63, 11a64, 11a144, 11a299, 11a320, 11a329, 12a35, 12a37, 12a47, 12a75, 12a97, 12a152, 12a231, 12a254, 12a269, 12a289, 12a320, 12a331, 12a421, 12a443, 12a610, 12a653, 12a661, 12a763, 12a764, 12a795, 12a796, 12a798, 12a800, 12a812, 12a880, 12a938, 12a967, 12a974, 12a978, 12a983, 12a996.   \label{thm:newuk}
\end{theorem}

\begin{theorem}
    The following 13 crossing knots have $u(K) = 3$: 13a16, 13a48, 13a60, 13a65, 13a82, 13a150, 13a154, 13a170, 13a172, 13a209, 13a217, 13a231, 13a251, 13a419, 13a434, 13a438, 13a443, 13a571, 13a576, 13a584, 13a586, 13a589, 13a644, 13a645, 13a653, 13a661, 13a666, 13a670, 13a796, 13a859, 13a901, 13a933, 13a983, 13a1043, 13a1051, 13a1062, 13a1137, 13a1140, 13a1147, 13a1182, 13a1295, 13a1329, 13a1331, 13a1339, 13a1403, 13a1413, 13a1440, 13a1465, 13a1469, 13a1472, 13a1475, 13a1491, 13a1495, 13a1499, 13a1546, 13a1547, 13a1552, 13a1554, 13a1580, 13a1615, 13a1616, 13a1618, 13a1643, 13a1943, 13a2037, 13a2125, 13a2185, 13a2269, 13a2290, 13a2291, 13a2295, 13a2310, 13a2317, 13a2334, 13a2359, 13a2360, 13a2365, 13a2373, 13a2379, 13a2399, 13a2412, 13a2434, 13a2478, 13a2502, 13a2520, 13a2697, 13a2698, 13a2761, 13a2766, 13a2777, 13a2784, 13a2785, 13a2805, 13a2821, 13a2832, 13a2834, 13a2843, 13a2859, 13a2865, 13a2888, 13a2911, 13a3065, 13a3089, 13a3100, 13a3310, 13a3531, 13a3601, 13a3691, 13a3698, 13a3704, 13a3705, 13a3716, 13a3784, 13a3902, 13a3963, 13a3975, 13a3981, 13a4057, 13a4080, 13a4091, 13a4093, 13a4101, 13a4104, 13a4106, 13a4113, 13a4156, 13a4185, 13a4191, 13a4198, 13a4200, 13a4207, 13a4216, 13a4232, 13a4293, 13a4320, 13a4322, 13a4385, 13a4390, 13a4392, 13a4396, 13a4461, 13a4465, 13a4508, 13a4510, 13a4556, 13a4565, 13a4571, 13a4579, 13a4580, 13a4624, 13a4655, 13a4756, 13a4796, 13a4810, 13a4820, 13a4835, 13a4836.
    \label{thm:newuk13}
\end{theorem}

We have also improved lower bounds on $u(K)$ for 30 knots, as below;

\begin{theorem}
    The following knots have $u(K)\in\{3,4\}$: 11a354, 12a156, 12a392, 12a1037, 12a1097, 12a1113, 13a2129, 13a2767, 13a3176, 13a3595, 13a4088, 13a4196, 13a4214, 13a4233, 13a4304, 13a4454, 13a4493, 13a4554, 13a4773, 13a4781, 13a4790, 13a4813, 13a4840, 13a4842, 13a4851, 13a4853, 13a4855, 13a4869, 13a4871, 13a4872.
    \label{thm:newukbound}
\end{theorem}

We may also add a sign restriction on the following knots having unknotting number 2. 

\begin{theorem}
    If the following knots do have $u(K) = 2$, then they must be unknotted with two crossings of the same sign: 11a362, 11a363, 12a86, 12a150, 12a153, 12a161, 12a372, 12a376, 12a379, 12a414, 12a767, 12a989, 12a1103, 12a1118, 13a17, 13a51, 13a159, 13a175, 13a285, 13a338, 13a339, 13a365, 13a382, 13a473, 13a619, 13a800, 13a862, 13a998, 13a1150, 13a1259, 13a1332, 13a1361, 13a1406, 13a1476, 13a1542, 13a1548, 13a1553, 13a1581, 13a1609, 13a1788, 13a1868, 13a1887, 13a2228, 13a2327, 13a2328, 13a2386, 13a2396, 13a2422, 13a2436, 13a2451, 13a2711, 13a2795, 13a2822, 13a2882, 13a2886, 13a2897, 13a2917, 13a3174, 13a3175, 13a3274, 13a3451, 13a3452, 13a3496, 13a3597, 13a3692, 13a3701, 13a3706, 13a3808, 13a3813, 13a4072, 13a4102, 13a4154, 13a4208, 13a4278, 13a4337, 13a4338, 13a4401, 13a4515, 13a4705, 13a4843, 13a4856, 13a4873.
    \label{thm: writhe-restriction}
\end{theorem}

Our computations are supplied at \cite{Data} for the reader to verify, but below we give all the details necessary to reproduce Theorems \ref{thm:newuk} through \ref{thm: writhe-restriction}. 

\subsection*{Conventions} We denote the double branch cover of a knot $K$ by $\Sigma(K)$. As per \cite[Satz 6]{schubert1956knoten}, we denote the two-bridge knot corresponding to $L(p,q)$ by $K(p,q)$. In computing the correction terms we use the convention of \cite{Owens} to orient $L(p,q)$, that is, the opposite of \cite{ozsz}. Note that Theorem \ref{thm:owens} only applies to knots with non-negative signature, but this can always be arranged by taking mirror images. Since for any knot $u(K)=u(K^m)$, we suppose that $\sigma(K)\geq0$ always.

\subsection*{Acknowledgments}
Both authors are extremely grateful to Brendan Owens for providing immense feedback and a Maple notebook to compare results. We also thank him for his incredible hospitality and insights during our brief visit to the University of Glasgow. We would also like to thank Hans Boden for his unwavering support and invaluable guidance in compiling this note.

\hfill

\noindent\textbf{Background.} Currently on KnotInfo \cite{knotinfo}, the knots listed in Theorems \ref{thm:newuk} and \ref{thm:newuk13} all have $u(K)\in\{2,3\}$, and the knots listed in Theorem \ref{thm:newukbound} all have $u(K)\in\{2,3,4\}$. For each of these, we use Theorem 1 in \cite{Owens} to obstruct $u(K)=2$, and hence conclude that $u(K)=3$ (or $u(K)\in\{3,4\}$). The knots listed in Theorem \ref{thm: writhe-restriction} all have $u(K) \geq 2$. We now summarize Owens' result and the necessary background information.

Firstly, recall that given a Seifert matrix $V$ for a knot $K$, the \textbf{signature} of $K$ is defined to be the signature of $V+V^T$. This is commonly denoted $\sigma(K)$, and it gives a well-defined and useful knot invariant. Remark 3 in \cite{Giller} (also see Proposition 2.1 in \cite{CL}) tells us that the behaviour of $\sigma(K)$ under crossing changes is well understood. Indeed, if $K'$ is obtained from $K$ by changing a positive crossing, then $\sigma(K')=\sigma(K)$ or $\sigma(K)+2$. Alternatively, if $K'$ is obtained from $K$ by changing a negative crossing, then $\sigma(K')=\sigma(K)$ or $\sigma(K)-2$. Note that this further implies the known fact that $u(K)\geq|\sigma(K)|/2$ for any knot $K$.

A positive definite matrix $Q\in GL_r(\Z)$ presents a group $\Gamma_Q$ via the short exact sequence
$$0\longrightarrow\Z^r\overset{Q}{\longrightarrow}\Z^r\longrightarrow\Gamma_Q\longrightarrow0.$$
Note that since $Q$ is positive-definite, it has rank $r$ and the group $\Gamma_Q$ will always be finite as it is a quotient of the free abelian group of rank $r$.

A \textbf{characteristic covector} for $Q$ is an element $E$ of $\Z^r$ which is congruent to the diagonal of $Q$ modulo 2. We denote the set of these by $\text{Char}(Q)=\{E\in\Z^r~|~E_i\equiv Q_{ii}~(\text{mod } 2)\}.$ Now define a function $m_Q:\Gamma_Q\to\Q$ by
$$m_Q(g)=\min\left\{\frac{E^TQ^{-1}E-r}{4}~|~E\in\text{Char}(Q),~[E]=g\right\}.$$
Since $Q$ is positive definite, this minimum must exist as the quantity $E^TQ^{-1}E$ is always non-negative. It is now possible to state Theorem 1 of \cite{Owens}:

\begin{theorem}
	[\cite{Owens}, Theorem 1]
	
	Let $K$ be a knot which may be unknotted by changing $p$ positive and $n$ negative crossings, where $n=\sigma(K)/2$ and $p+n=2$. Then there exists a positive-definite matrix
	$$Q=\bmat{m_1&1&a&0\\1&2&0&0\\a&0&m_2&1\\0&0&1&2},$$
	with $\det Q=\det K$, $0\leq a<m_1\leq m_2$, and exactly $n$ of $\{m_1,m_2\}$ are even; and a group isomorphism $\phi:\Gamma_Q\to\textup{Spin}^c(\Sigma(K))$ such that for all $g\in\Gamma_Q$,
	$$m_Q(g)\geq d(\Sigma(K), \phi(g)) ~~~~~~~~\text{ and }~~~~~~~~ m_Q(g)\equiv d(\Sigma(K), \phi(g))~(\text{mod } 2).$$
	
	\label{thm:owens}
\end{theorem}

When $\Sigma$ is a lens space, the following proposition of Ozsv\'ath and Szab\'o allows for a simple computation of the correction terms used above. 

\begin{proposition}
    [\cite{ozsz}, Proposition 4.8] Let $\Sigma = L(p,q)$ be a lens space. The correction terms of $\text{Spin}^c(\Sigma)$, with a canonical cyclic ordering, may be computed recursively using the following formula
    \begin{align*}
        \text{Base: } d(L(p,0), i) = 0 \\
        d(L(p,q), i) = \left(\frac{pq - (2i + 1 - p - q)^2}{4pq} \right) - d (L(q,r), j),
    \end{align*}
    where $i$ is the index of the ordering, $r = p \mod q$, and $j = i \mod q$.
    \label{prop:ozsz}
\end{proposition}

It is worth noting here that the canonical cyclic ordering of $i$ gives the group structure of $\text{Spin}^c(\Sigma)\cong H^2(\Sigma)$ up to a shift and potential sign correction (see \cite[Proposition 4.2]{ozsz} and the proof of Lemma 2.1 in \cite{owens:strle}). As two-bridge knots are double branch covered by lens spaces (see \cite[Satz 6]{schubert1956knoten} or \cite[Proposition 12.3]{lenscover}), this proposition may be combined with Theorem \ref{thm:owens} to yield Theorems \ref{thm:newuk} through \ref{thm: writhe-restriction} for all of the two-bridge knots listed. 

If $K$ is a general alternating knot, Owens \cite[Theorem 4.1]{Owens} uses a method due to Ozsv\'ath and Szab\'o \cite[Proposition 3.2]{ozsz:unknotting} for computing the correction terms using the positive definite Goeritz matrix of an alternating diagram of $K$. This method combined with Theorem \ref{thm:owens} gives the lower bound for all other alternating knots listed in Theorems \ref{thm:newuk} through \ref{thm: writhe-restriction}.

Our worked example highlights the two-bridge case as computing the correction terms via the Goeritz matrix is well documented in \cite[Theorem 4.1, Corollary 2]{Owens} and \cite[Proposition 3.2]{ozsz:unknotting}.

\hfill

\noindent\textbf{Proof of Theorems \ref{thm:newuk} through \ref{thm: writhe-restriction} :} We provide the details for $12a796=K(57,11)$. Note that a similar procedure can be used to obstruct unknotting number 2 for all of the knots listed. In the case of Theorem \ref{thm: writhe-restriction} we obstruct unknotting with $p = 1,~n = 1$ crossing changes. As all of those knots have signature 2, we are left with $n = 2,~p = 0$ being the only possible unknotting configuration. Explicit details of these obstructions may be found at \cite{Data}.

Let $K=12a796$. One can easily compute that $\sigma(K)=4$, so it follows that $u(K) \geq 2$. Moreover, notice that $K$ may be unknotted by performing the three crossing changes indicated in Figure \ref{fig:12a}. This tells us that $u(K) \in \{2,3\}$.

\begin{figure}[ht]
    \centering
    \includegraphics[width=0.4\linewidth]{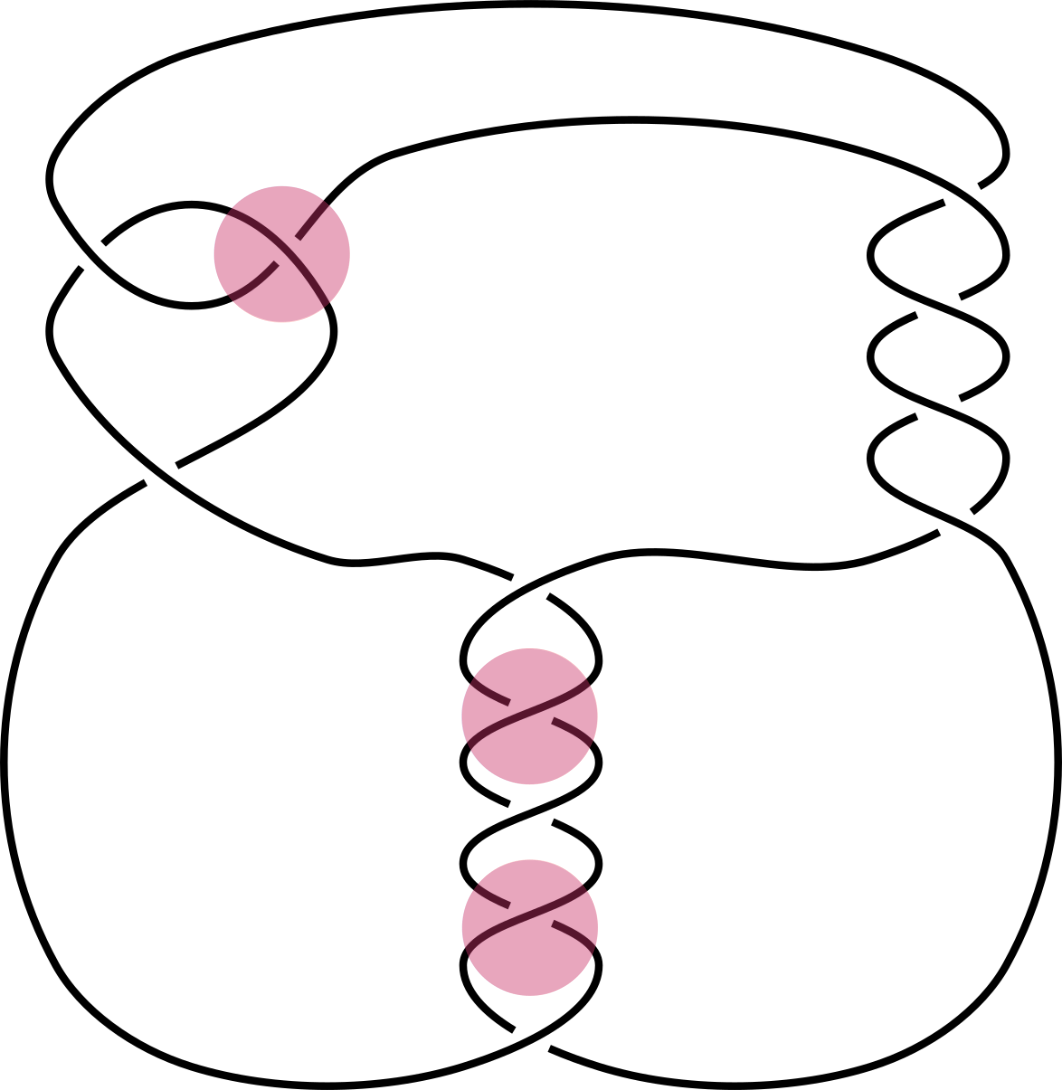}
    \caption{$u(12a796)\leq3$.}
    \label{fig:12a}
\end{figure}

Now suppose that $u(K)=2$ and note that if this were the case, then both crossing changes would need to be on negative crossings so that we end with signature 0. Hence, it must be that $K$ is unknotted by $p=0$ positive crossing changes and $n=2$ negative crossing changes, so we satisfy the hypotheses of Theorem \ref{thm:owens}.

Recall that the two-bridge knot $K(p,q)$ is double-branch covered by the lens space $L(p,q)$. In our case, we need to look at $L(57,11)$. By Proposition \ref{prop:ozsz}, or via a Goeritz matrix of $K$, we compute that the correction terms for Spin$^c(L(57,11))$ with respect to the group structure are as follows
\begingroup
\renewcommand*{\arraystretch}{1.5}
    $${ d= \left\{\begin{array}{rrrrrrrrrrrrr} -1 & \red{-\frac{83}{57}} & \red{-\frac{47}{57}} & -\frac{21}{19} & \red{-\frac{17}{57}} & \red{-\frac{23}{57}} & \frac{11}{19} & \frac{37}{57} & -\frac{11}{57} & \frac{1}{19} & \red{-\frac{35}{57}} & -\frac{11}{57} & -\frac{13}{19} \\ -\frac{5}{57} & \red{-\frac{23}{57}} & \frac{7}{19} & \frac{13}{57} & \frac{67}{57} & \frac{23}{19} & \frac{1}{3} & \frac{31}{57} & -\frac{3}{19} & \frac{13}{57} & \red{-\frac{17}{57}} & \frac{5}{19} & -\frac{5}{57} \\ \frac{37}{57} & \frac{9}{19} & \frac{79}{57} & \frac{79}{57} & \frac{9}{19} & \frac{37}{57} & -\frac{5}{57} & \frac{5}{19} & -\frac{17}{57} & \frac{13}{57} & -\frac{3}{19} & \frac{31}{57} & \frac{1}{3} \\ \frac{23}{19} & \frac{67}{57} & \frac{13}{57} & \frac{7}{19} & -\frac{23}{57} & -\frac{5}{57} & -\frac{13}{19} & -\frac{11}{57} & -\frac{35}{57} & \frac{1}{19} & -\frac{11}{57} & \frac{37}{57} & \frac{11}{19} \\ -\frac{23}{57} & -\frac{17}{57} & -\frac{21}{19} & -\frac{47}{57} & -\frac{83}{57} \end{array}\right\} }$$
    \endgroup 

Since $12a796$ has determinant 57, we now seek all matrices $Q$ of the form specified in Theorem \ref{thm:owens} with determinant 57. This is equivalent to finding all integer solutions to the equation
\begin{equation}
    (2m_1-1)(2m_2-1)=57+4a^2,
\end{equation}
which satisfy $0\leq a<m_1\leq m_2$ and $m_1,m_2$ both even. First suppose that $a\geq15$. Then since $m_2\geq m_1\geq a+1$, we find by (1) that $57+4a^2\geq(2a+1)^2$, which implies that $57\geq4a+1$. But this is impossible since $4a+1>60$, so there are no solutions when $a\geq15$. Note additionally that we have no solutions when $57+4a^2$ is prime, since the only way that $(2m_1-1)(2m_2-1)$ can be prime is if $m_1=1$ or $m_2=1$. Neither is possible since $m_1$ and $m_2$ must be even. Now consider $0\leq a\leq 14$, and note that $57+4a^2$ has integer sequence
$$3\cdot19,~61,~73,~3\cdot31,~11^2,~157,~3\cdot67,~11\cdot23,~313,~3\cdot127,~457,~541,~3\cdot211,~733,~29^2.$$
One can check that $57+4a^2$ is composite only when $a\in\{0,3,4,6,7,9,12,14\}$. Hence, there are no solutions outside of these $a$ values. Moreover, $a\notin\{3,6,9,12\}$ by the prime factorizations of $57+4a^2$ in these cases, all of which require $m_1=2<a$. Similarly, $a\neq7$ since that requires $m_1=6<a$. Lastly, $a\neq14$ since this requires $m_1=m_2=15$, which is not even. Hence, $a=0, m_1=2, m_2=10$ and $a=4,m_1=m_2=6$, are the only solutions, and
$$Q_1=\bmat{2&1&0&0\\1&2&0&0\\0&0&10&1\\0&0&1&2},\hspace{50pt}Q_2=\bmat{6&1&4&0\\1&2&0&0\\4&0&6&1\\0&0&1&2}$$
are the only possible matrices from Theorem \ref{thm:owens}.
In both cases, we find that $\Gamma_{Q}\cong\Z_{57}$, and we compute $m_{Q_1}$ and $m_{Q_2}$ as
\begingroup
\renewcommand*{\arraystretch}{1.5}
$${ 	 m_{Q_1}=
 \left\{\begin{array}{rrrrrrrrrrrrr} -1 & -\frac{13}{57} & \frac{5}{57} & -\frac{1}{19} & \frac{77}{57} & \frac{131}{57} & \frac{15}{19} & \frac{47}{57} & \frac{23}{57} & -\frac{9}{19} & \frac{11}{57} & \frac{23}{57} & \frac{3}{19} \\ \frac{83}{57} & \frac{131}{57} & \frac{13}{19} & \frac{35}{57} & \frac{5}{57} & -\frac{17}{19} & -\frac{1}{3} & -\frac{13}{57} & -\frac{11}{19} & \frac{35}{57} & \frac{77}{57} & \frac{31}{19} & \frac{83}{57} \\ \frac{47}{57} & -\frac{5}{19} & \frac{11}{57} & \frac{11}{57} & -\frac{5}{19} & \frac{47}{57} & \frac{83}{57} & \frac{31}{19} & \frac{77}{57} & \frac{35}{57} & -\frac{11}{19} & -\frac{13}{57} & -\frac{1}{3} \\ -\frac{17}{19} & \frac{5}{57} & \frac{35}{57} & \frac{13}{19} & \frac{131}{57} & \frac{83}{57} & \frac{3}{19} & \frac{23}{57} & \frac{11}{57} & -\frac{9}{19} & \frac{23}{57} & \frac{47}{57} & \frac{15}{19} \\ \frac{131}{57} & \frac{77}{57} & -\frac{1}{19} & \frac{5}{57} & -\frac{13}{57} \end{array}\right\}
 } \\ $$

 $${ 	
 m_{Q_2} = \left\{\begin{array}{rrrrrrrrrrrrr} -1 & -\frac{35}{57} & \frac{31}{57} & \frac{9}{19} & \frac{67}{57} & \frac{37}{57} & \frac{17}{19} & -\frac{5}{57} & -\frac{17}{57} & \frac{5}{19} & -\frac{23}{57} & -\frac{17}{57} & \frac{11}{19} \\ \frac{13}{57} & \frac{37}{57} & -\frac{3}{19} & -\frac{11}{57} & \frac{31}{57} & \frac{1}{19} & \frac{1}{3} & -\frac{35}{57} & -\frac{15}{19} & -\frac{11}{57} & \frac{67}{57} & \frac{25}{19} & \frac{13}{57} \\ -\frac{5}{57} & \frac{7}{19} & -\frac{23}{57} & -\frac{23}{57} & \frac{7}{19} & -\frac{5}{57} & \frac{13}{57} & \frac{25}{19} & \frac{67}{57} & -\frac{11}{57} & -\frac{15}{19} & -\frac{35}{57} & \frac{1}{3} \\ \frac{1}{19} & \frac{31}{57} & -\frac{11}{57} & -\frac{3}{19} & \frac{37}{57} & \frac{13}{57} & \frac{11}{19} & -\frac{17}{57} & -\frac{23}{57} & \frac{5}{19} & -\frac{17}{57} & -\frac{5}{57} \\ \frac{17}{19} & \frac{37}{57} & \frac{67}{57} & \frac{9}{19} & \frac{31}{57} & -\frac{35}{57} \end{array}\right\} }$$
 
\endgroup

where the orders of these lists are determined by the ordering of $\Z_{57}$ with generator $\gamma$ (i.e. the $i$th entry is $m_Q(\gamma^i)$).

To contradict Theorem \ref{thm:owens}, we must show in both cases that there is no isomorphism $\phi:\Gamma_Q\to\text{Spin}^c(\Sigma(K))$ with the specified properties. In particular, we will show that for all possible $\phi$, there exists some $g\in\Gamma_Q$ such that $m_Q(g)<d(\Sigma(K), \phi(g))$. 

As both $\Gamma_Q$ and $\text{Spin}^c(\Sigma(K))$ are cyclic groups of order $57$, we may instead think of $\phi$ as an automorphism from $\Z_{57}$ to itself, and thus multiplication by some some integer $x$ coprime to $57$. Furthermore, the palindromic nature of $d : \Z_{57} \rightarrow \Q$ tells us if multiplication by $x$ is obstructed, then multiplication by $-x$ is as well. Thus we may also assume $x \leq 57/2$.

Let us first obstruct $Q_1$. As $m_{Q_1} (1) = -\frac{13}{57}$, we need $x$ so that $d(x) \leq -\frac{13}{57}$ and differs from $-\frac{13}{57}$ by a multiple of 2. The \red{highlighted $d$} values above correspond to generators with $d$ values less than or equal to $-\frac{13}{57}$. However, each potential generator fails the mod $2$ condition, so $Q_1$ is obstructed, and we move to $Q_2$.

In the case of $Q_2$ we may not obstruct all $\phi$ as easily as multiplication by $10$ gives an isomorphism with $m_{Q_2}(1) = -\frac{35}{57} = d(10)$. We may now go further through the relevant $m_{Q_2}$ and $d$ values of the isomorphism until this isomorphism is obstructed. The list below offers such an obstruction as $m_{Q_2}(20) = -\frac{35}{57} < d(200 \equiv 29) = \frac{79}{57}$.

\begin{verbatim}
    mQ2[i]  =   [-35/57, 31/57, 9/19, 67/57, 37/57, 17/19, -5/57, -17/57, 5/19, -23/57,
                 -17/57, 11/19, 13/57, 37/57, -3/19, -11/57, 31/57, 1/19, 1/3, -35/57]

d[10i mod 57] = [-35/57, 31/57, 9/19, 67/57, 37/57, -21/19, -5/57, -17/57, 5/19, -23/57, 
                 -17/57, 11/19, 13/57, 37/57, -3/19, -11/57, -83/57, 1/19, 1/3, 79/57]               
\end{verbatim}

In either case we have all possible isomorphisms obstructed and so $u(12a796) = 3$ as desired.
\qed

\begin{remark}
    In \cite{brittenham}, Brittenham showed the knots 11a237, 11a337, and 11a359 have $u(K)=3$ by showing that they are 1 crossing change away from 11a365, which has $u(K)=4$ as per \cite{Owens}. He also showed 11a361 has $u(K) \in \{3,4\}$ for the same reason. Our computations give these results as a direct application of Owens' Theorem 1 in that paper.
\end{remark}

\section*{Appendix and Python functions}
We give a SageMath \cite{sagemath} and SnapPy \cite{SnapPy} implementation of the process described above for a given two-bridge or alternating knot. A full notebook to use with KnotInfo data is available at \cite{Data}.
{\tiny
\begin{verbatim}
import itertools
from snappy import Link

All_Qs = dict()
Knot_Names = dict()

def load_pickle_file(file_path):
    with open(file_path, 'rb') as file:
        return pickle.load(file)

def dump_pickle_file(file_path, file_data):
    with open(file_path,"wb") as file:
        pickle.dump(file_data, file)
    return True
    
def d_invariant_of_two_bridged_knots(p, q):
    def d(p, q, i):
        return 0 if q == 0 else 1/4 - 1/4 * (2 * i + 1 - p - q)^2 / (p*q) - d(q, p % q, i %q)
    original_dlist = [d(p, q, i) for i in range(0, p)]
    # Rotate the resulting list to align with the cyclic group structure
    k = (q-1)/2 if q%2 == 1 else (p+q-1)/2
    dlist = vector(original_dlist[k:] + original_dlist[:k])
    return dlist if dlist[0] < 0 else -1 * dlist

def convert_lattice_to_group_word(lattice_element, group):
    gens = group.gens()
    group_element = gens[0]^lattice_element[0]
    for i in range(1, len(lattice_element)):
        group_element *= gens[i]^lattice_element[i]
    return group_element

def quotient_by_lattice(Q):
    r = rank(Q)
    G = FreeGroup(rank(Q))
    gens = G.gens()
    commutators = [gens[i]*gens[j]*gens[i]^(-1)*gens[j]^(-1) for i in range(r-1) for j in range(i+1, r)]
    lattice_relation_vectors = (Q).columns()
    return G.quotient(commutators + [convert_lattice_to_group_word(x,G) for x in lattice_relation_vectors])

def get_characteristic_covectors(Q):
    each_component_options = [list(range(-Q[i][i], Q[i][i]-1, 2)) for i in range(rank(Q))]
    return itertools.product(*each_component_options)
                   
def make_intersection_form(m2, m1, a):
    return matrix([[m1, 1, a, 0], [1, 2, 0, 0], 
                   [a, 0, m2, 1], [0, 0, 1, 2]])

def find_all_possible_intersection_forms(detK, n = 2):
    possibleQ = []
    for a in range(0, detK//4 + 1):
        for k1 in divisors(detK + 4 * a^2):
            k2 = (detK + 4 * a^2)/k1
            if k1 <= k2 and sum([x%4 for x in [k1,k2]]) == 3 * n + 1 * (2-n):
                m1 = (k1 + 1)/2
                m2 = (k2 + 1)/2
                if a < m1:
                    possibleQ += [make_intersection_form(m2, m1, a)]
    return possibleQ

def get_group_element(covector, Q):
    group = quotient_by_lattice(Q)
    return convert_lattice_to_group_word(covector, group)

def get_dlist_Q(Q):
    unique_char_cvectors = []
    generators = []
    for char_covector in get_characteristic_covectors(Q):
        m_covector = matrix(char_covector)
        current_d_value = (m_covector * Q^(-1) * m_covector.transpose() - rank(Q))/4
        current_group_element = get_group_element(char_covector, Q)
        if current_group_element.order() == det(Q):
            generators += [current_group_element]
        unique = True
        for i in range(len(unique_char_cvectors)):
            (group_element, d_value) = unique_char_cvectors[i]
            if current_group_element == group_element:
                unique = False
                unique_char_cvectors[i] = (group_element, min([d_value, current_d_value]))
                break
        if unique:
            unique_char_cvectors += [(current_group_element, current_d_value)]
    new_dlist = []
    if len(generators) > 0:
        generator = generators[0]
        for i in range(det(Q)):
            current_group_element = generator^i
            for (group_element, d_value) in unique_char_cvectors:
                if group_element == current_group_element:
                    new_dlist += [d_value[0][0]]
                    break
    return new_dlist

def pretty_print(original_dlist):
    new_list = []
    length = len(original_dlist)
    for i in range(length//15 + 1):
        new_list += [[]]
        for j in range(15):
            pointer = i*15 + j
            if pointer >= length:
                new_list[i] += [0]
            else:
                new_list[i] += [original_dlist[pointer]]
    print(matrix(new_list))

def obstruct(dlist, Q):
    # Note if the group is cyclic we may assume an automorphism of Z_p
    Q_dlist = get_dlist_Q(Q)
    p = Integer(det(Q))
    if len(Q_dlist) != p:
        print ("Obstructed by Z^n/Q not being cyclic")
        return True
  
    (shift_for_simplicity, _) = find_minimal_generator(Q_dlist)
    isomorphic_Q_dlist = vector(isomorphic_dlist(Q_dlist, shift_for_simplicity))
    pretty_print(isomorphic_Q_dlist)
    
    for possible_isomorphism_shift in p.coprime_integers((p+1)/2):
        isomorphism_attempt_d_values = []
        isomorphism = vector(isomorphic_dlist(dlist, possible_isomorphism_shift))
        for group_index in range(1, p):
            isomorphism_attempt_d_values += [isomorphism[group_index]]
            difference_in_values = isomorphism[group_index] - isomorphic_Q_dlist[group_index]
            obs_diff_mod_2 = not (difference_in_values / 2).is_integer()
            if isomorphism[group_index] > isomorphic_Q_dlist[group_index] or obs_diff_mod_2:
                break
        if len(isomorphism_attempt_d_values) < p - 1:
            if len(isomorphism_attempt_d_values) > 1:
                print(list(isomorphic_Q_dlist[1: len(isomorphism_attempt_d_values) + 1]))
                print(possible_isomorphism_shift, isomorphism_attempt_d_values)
        else:
            print("Unobstructed")
            print(possible_isomorphism_shift)
            return False
    return True

def find_minimal_generator(MCQ_list):
    p = Integer(len(MCQ_list))
    gens = [(i, MCQ_list[i]) for i in p.coprime_integers(p)]
    return min(gens, key = lambda generator: generator[1])

def isomorphic_dlist(dlist, shift):
    p = len(dlist)
    assert gcd(shift, p) == 1
    return [dlist[i*shift % p] for i in range(p)]

def obstruct_alternating(name, expected_signature = 4):
    print(f"{name}")
    L = Link(name)
    assert L.signature() in [-expected_signature, expected_signature]
    assert L.is_alternating()
    signature = abs(L.signature())
    G = L.goeritz_matrix()
    G = G if G.is_positive_definite() else L.mirror().goeritz_matrix()
    assert G.is_positive_definite()
    detK = abs(det(G))
    dlist = vector(get_dlist_Q(G))
    if len(dlist) != detK:
        print("Not cyclic will skip for now")
        return
    if not dlist[0] in [-signature/4, signature/4]:
        print("Bad Goeritz matrix. Should check")
        return
    dlist = -dlist if dlist[0] == signature/4 else dlist # correct possible sign error 
    print(f"Correction terms")
    pretty_print(dlist)
    if detK in All_Qs:
        QS = All_Qs[detK]
    else:
        QS = find_all_possible_intersection_forms(detK, n = signature/2)
    All_Qs[detK] = QS
    unobstructed = False
    for Q in QS:
        print("Testing obstruction for")
        print(f"m1: {Q[0][0]}, m2: {Q[2][2]},a: {Q[0][2]}")
        unobstructed = not obstruct(dlist, Q)
        if unobstructed:
            break
    print("---------------------------")
    if not unobstructed:
        print(f"SUCCESS for name = {name}")
        Knot_Names[name] = True
    print("---------------------------")

def obstruct_twobridge(p,q, name):
    print(f"Obstructing p = {p} | q = {q}")
    dlist = d_invariant_of_two_bridged_knots(p,q)
    print(f"Correction terms")
    pretty_print(dlist)
    if p in All_Qs:
        QS = All_Qs[p]
    else:
        QS = find_all_possible_intersection_forms(p)
    All_Qs[p] = QS
    unobstructed = False
    for Q in QS:
        print("Testing obstruction for")
        print(f"m1: {Q[0][0]}, m2: {Q[2][2]},a: {Q[0][2]}")
        unobstructed = not obstruct(dlist, Q)
        if unobstructed:
            break
    print("---------------------------")
    if not unobstructed:
        print(f"SUCCESS for p = {p} | q = {q} | name = {name}")
        Knot_Names[name] = True
    print("---------------------------")
\end{verbatim}
}


\end{document}